\documentclass{elsarticle}

\usepackage[english]{babel} %francais, polish, spanish, \dots
\usepackage[T1]{fontenc}
\usepackage[ansinew]{inputenc}
\usepackage{amssymb}
\usepackage{algorithm,algorithmic}
\usepackage{url}
\usepackage{comment}
\usepackage{tabularx}

\usepackage{hyperref}

\usepackage{adjustbox}

\addto\captionsenglish{}
\usepackage{graphicx} %%For loading graphic files

\usepackage{caption}
\usepackage{subcaption}

\usepackage{tikz}
\usetikzlibrary[patterns, decorations.pathreplacing,calc]

\usepackage{amsmath}
\usepackage{amsthm}
\usepackage{amsfonts}
\usepackage{mathrsfs}

\newtheorem{theorem}{Theorem}

\newtheorem{lemma}[theorem]{Lemma}
\newtheorem{proposition}[theorem]{Proposition}
\newtheorem{definition}[theorem]{Definition}

\newtheorem{observation}[theorem]{Observation}

\newtheorem*{thm:up}{Theorem \ref{thm:upper}}
\newtheorem*{thm:lower}{Theorem \ref{thm:lower}}

\begin{document}
\newcommand\m{\mathcal}

\begin{frontmatter}

%% Title Page %%%%%%%%%%%%%%%%%%%%%%%%%%%%%%%%%%%%%%%%%%%%%%%
%% ==> Write your text here or include other files.

%% The simple version:
\title{Bounding the Number of Minimal Transversals in Tripartite 3-Uniform Hypergraphs}
\author{Alexandre Bazin$^1$, Laurent Beaudou$^2$, Giacomo Kahn$^3$, Kaveh Khoshkhah$^4$}
\address{(1) LORIA, Universit\'e de Lorraine -- CNRS -- Inria, Nancy, France\\
(2) Higher School of Economics, Moscow, Russian Federation\\
(3) Universit\'e Lumi\`ere Lyon 2, Laboratoire DISP, Bron, France\\
(4) Institute of Computer Science, University of Tartu, Tartu, Estonia\\\texttt{contact@alexandrebazin.com, lbeaudou@hse.ru, giacomo.kahn@univ-lyon2.fr, khoshkhah@theory.cs.ut.ee}}
%\date{} %%If commented, the current date is used.
%\maketitle

\begin{abstract}
We focus on the maximum number of minimal transversals in $3$-partite $3$-uniform hypergraphs on $n$ vertices. Those hypergraphs (and their minimal transversals) are commonly found in database applications. In this paper we prove that this number grows at least like $1.4977^{n}$ and at most like $1.5012^n$.
\end{abstract}
\end{frontmatter}

\section{Introduction}

%TODO : réécrire l'introduction

Hypergraphs are a generalization of graphs where edges may have arities different than 2.
They were formalized by Berge in the seventies~\cite{berge1973graphs}.
Formally, a \emph{hypergraph} $\m H$ is a pair $(V,\m E)$, where $V$ is a set of vertices and $\m E$ a family of subsets of $V$ called hyperedges.
In the following, we suppose that $V = \bigcup_{e\in\mathcal E}e$, \textit{i.e.}, the hypergraphs that we handle have no isolated vertices.
The number of vertices of a hypergraph is called its order.
When all the hyperedges of a hypergraph have the same arity $p$, we call it a \emph{$p$-uniform hypergraph}.
%In this sense, graphs are exactly the 2-uniform hypergraphs.
When the set of vertices of a hypergraph can be partitioned into $k$ sets such that every edge intersects each part at most once, the hypergraph is called \emph{$k$-partite}.
In the following, we are interested in tripartite, 3-uniform hypergraphs, sometimes known as $(3,3)$-hypergraphs.
A subset of the vertices of a hypergraph $\m H$ is a \emph{transversal} of $\m H$ if it intersects every edge of $\m H$ and is said to be a \emph{minimal transversal} if none of its proper subsets is a transversal.

\medskip
The problem of enumerating the minimal transversals of a given hypergraph has been extensively studied~\cite{eiter1995identifying, fredman1996complexity, gunopulos1997data}. It is an important problem in theoretical computer science as it is equivalent to the dualization of monotone Boolean functions and the enumeration of maximal independent sets~\cite{eiter2008computational, kante2014enumeration}. As such, it has many real-world applications in artificial intelligence~\cite{eiter2002hypergraph}, biology~\cite{damaschke2006parameterized}, mobile communication~\cite{sarkar1998hypergraph} and data mining~\cite{gunopulos1997data,stavropoulos2016transversal, adaricheva2017discovery, van2019query}, among others.
Since the number of minimal transversals of a hypergraph may be exponential in the order of the hypergraph, the complexity of enumeration algorithms is often expressed as a function of the size of both their input and their output~\cite{strozecki2019enumeration}.
The maximum number of minimal transversals in a class of hypergraph can then be an important parameter for the real world applications of those theoretical algorithmic evaluations.

\medskip
Here, we are interested in the number of minimal transversals that arise in tripartite 3-uniform hypergraphs. Such hypergraphs correspond to 3-dimensional Boolean datasets, e.g. data describing \emph{objects} by sets of \emph{attributes} under different \emph{conditions}, and the enumeration of their minimal transversals is involved in the mining of different patterns such as triadic concepts~\cite{lehmann1995triadic,ignatov2015triadic} and association rules~\cite{missaoui2011mining}.
We denote by $f_3(n)$ the maximum number of minimal transversals in such hypergraphs of order $n$.

\medskip
In 1965, Moon and Moser~\cite{moon1965cliques} provided a construction of a graph of order $n$ that contains $3^{n/3}$ independent sets. This number is reached by using a disjoint union of $K_3s$.
The same construction can be adapted to $(3,3)$-hypergraphs as illustrated in Figure~\ref{fig:MnMs33sisi}: a set of disjoint 3-edges will form a $(3,3)$-hypergraph with $3^{n/3} = 1.4422^n$ minimal transversals, hence $f_3(n)\geq 1.4422^n$.

\begin{figure}[h]
\centering
\begin{subfigure}[b]{0.45\textwidth}
\adjustbox{max width = \textwidth}{
\begin{tikzpicture}[every node/.style={fill,circle,inner sep=0pt,minimum size=4pt}]

\node (a) at (0,0) {};
\node (b) at (-1,1) {};
\node (c) at (-1,-1) {};
\node (d) at (1.2,1) {};
\node (e) at (1.2,-1) {};
\node (f) at (2.2,0) {};
\node[draw=none, fill=none] (g) at (4,0) {\huge{$\dots$}};
\node (h) at (7,0) {};
\node (i) at (6,1) {};
\node (j) at (6,-1) {};

\draw (a) -- (b);
\draw (a) -- (c);
\draw (b) -- (c);
\draw (d) -- (f);
\draw (d) -- (e);
\draw (e) -- (f);
\draw (h) -- (i);
\draw (h) -- (j);
\draw (i) -- (j);

\end{tikzpicture}
}\hfill
\end{subfigure}
\begin{subfigure}[b]{0.45\textwidth}
\adjustbox{max width = \textwidth}{
\begin{tikzpicture}[every node/.style={fill,circle,inner sep=0pt,minimum size=4pt}]

\node (a) at (0,0) {};
\node (b) at (-1,1) {};
\node (c) at (-1,-1) {};
\node (d) at (1.2,1) {};
\node (e) at (1.2,-1) {};
\node (f) at (2.2,0) {};
\node[draw=none, fill=none] (g) at (4,0) {\huge{$\dots$}};
\node (h) at (7,0) {};
\node (i) at (6,1) {};
\node (j) at (6,-1) {};

  \draw ($(a)+(0.3,0)$)
    to[out=90,in=0] ($(b)+(0,0.3)$)
    to[out=180,in=180] ($(c)+(0,-0.3)$)
    to[out=0,in=-90] ($(a)+(0.3,0)$);

  \draw ($(f)+(0.3,0)$)
    to[out=90,in=0] ($(d)+(0,0.3)$)
    to[out=180,in=180] ($(e)+(0,-0.3)$)
    to[out=0,in=-90] ($(f)+(0.3,0)$);
    
  \draw ($(h)+(0.3,0)$)
    to[out=90,in=0] ($(i)+(0,0.3)$)
    to[out=180,in=180] ($(j)+(0,-0.3)$)
    to[out=0,in=-90] ($(h)+(0.3,0)$);
\end{tikzpicture}
}
\end{subfigure}
\caption{Moon and Moser's construction (left) and its analogue for $(3,3)$-hypergraphs (right).\label{fig:MnMs33sisi}}
\end{figure}

\medskip
Moreover, let $\m H$ be a $(3,3)$-hypergraph which vertices are partitioned into three sets $S_1$, $S_2$ and $S_3$ that each intersects each hyperedge at most once. For a given vertex set $X\subseteq S_1\cup S_2$, there is only one possible $Y\subseteq S_3$ such that $X\cup Y$ is a minimal transversal.
By supposing that all subsets of $S_1\cup S_2$ can appear in minimal transversals, we obtain $2^{n-|S_3|}$ minimal transversals. If we choose the smallest of the three sets to be $S_3$, then $S_3$ has at most $n/3$ vertices and so we obtain an upper bound of $2^{n/3}\times2^{n/3}=4^{n/3}\approx 1.5874^n$ minimal transversals.

\medskip

Thus, for any integer $n$, $1.4422^n\leq f_3(n)\leq 1.5874^n$.
In this paper, we improve those bounds through two theorems.

\begin{theorem}
\label{thm:lower}
There exists a constant $c$ such that for any integer $n\geq 15$, $$f_3(n)\geq c1.4977^n.$$
\end{theorem}

\begin{theorem}
\label{thm:upper}
For any integer $n$, $$f_3(n)\leq 1.5012^n.$$
\end{theorem}

We prove Theorem~\ref{thm:lower} through a construction based on a hypergraph on fifteen vertices found \textit{via} computer search. The proof of Theorem~\ref{thm:upper} relies on a technique introduced by Kullman~\cite{DBLP:journals/tcs/Kullmann99} and used by Lonc and Truszczy\'nski~\cite{DBLP:journals/dm/LoncT08} on rank 3 hypergraphs.
This class contains $(3,3)$-hypergraphs but is much larger, and the bound they obtain is greater than the trivial one implied by the tripartition in our case.
This proof technique resembles measure and conquer, an approach used in the analysis of exact exponential-time algorithms, see for example Fomin, Grandoni and Kratsch~\cite{DBLP:journals/jacm/FominGK09}.

%%%%%%%%%%%%%%%%%%%%%%%%%%%%%%%%%%%%%%%%%%%%%%%%%%%%%%%%%%%%%%%%%
%%%%%%%%%%%%%%%%%%%%%%%%%%%%%%%%%%%%%%%%%%%%%%%%%%%%%%%%%%%%%%%%%
%%%%%%%%%%%%%%%%%%%%%%%%%%%%%%%%%%%%%%%%%%%%%%%%%%%%%%%%%%%%%%%%%
\section{Lower Bound}

We consider tripartite 3-uniform hypergraphs.
In this section, we improve the lower bound of $1.4422^n$ for $f_3(n)$
by exhibiting a construction that reaches $c1.4977^n$ minimal transversals, where $c$ is a constant.
To this end, we first make an observation that allows us to multiply the number of minimal transversals while only summing the orders of hypergraphs.

\medskip

\begin{observation}
\label{lemma:mult}
let $\m H_1$ and $\m H_2$ be two hypergraphs of order $n_1$ and $n_2$ with, respectively, $t_1$ and $t_2$ minimal transversals. Then, the disjoint union of $\m H_1$ and $\m H_2$ has $n_1+n_2$ vertices and exactly $t_1t_2$ minimal transversals. To put it into words, we sum the orders while multiplying the number of minimal transversals.
%Let $k$ be an integer and, for each integer $i$ between $1$ and $k$, let $\m H_i$ be a hypergraph of order $n_i$ with $t_i$ minimal transversals.
%Then the hypergraph $\m H$ obtained by disjoint union of the $k$ hypergraphs is of order $\sum_{i=1}^k n_i$ and has $\prod_{i=1}^k t_i$ minimal transversals.
%Moreover, if for every $i$, $\m H_i$ is tripartite and 3-uniform, then $\m H$ is tripartite and 3-uniform.
\end{observation}
%\begin{proof}
%Since $\m H$ is the disjoint union of hypergraphs $\m H_i$, the minimal transversals of $\m H$ are exactly the sets of vertices resulting from the union of one minimal transversal from each $\m H_i$. Hence, $\m H$ has $\prod_{i=1}^k t_i$ minimal transversal.
%The disjoint union of the $\m H_i$, for all $i$ between 1 and $k$, is still 3-uniform because .
%\end{proof}

\medskip

A computer search of the space of small hypergraphs allows us to make the following observation.

\medskip

\begin{observation}
\label{claim:428}
There is a tripartite 3-uniform hypergraph on fifteen vertices with four hundred and twenty-eight minimal transversals.
\end{observation}

As such, \[f_3(15)\geq 428.\]

\medskip

The hypergraph mentioned in Observation~\ref{claim:428} is described in Figure~\ref{tab:edges} and Figure~\ref{tab:hyp}.
We denote this hypergraph by $\m H_{15}$.
Its minimal transversals can be computed using any available software, such as the one maintained by Takeaki Uno\footnote{\url{http://research.nii.ac.jp/~uno/dualization.html} with the data file at \url{http://giacomo.kahn.science/resources/H15.txt}}.

\begin{figure}[ht]
\centering
\begin{tabular}{ccccc}
$\{\alpha, 1,a\}$	&	$\{\beta, 1,b\}$ & $\{\gamma, 1,c\}$ & $\{\delta, 1,d\}$ & $\{\epsilon, 1,e\}$\\ 
$\{\alpha, 2,b\}$	&	$\{\beta, 2,c\}$ & $\{\gamma, 2,d\}$ & $\{\delta, 2,e\}$ & $\{\epsilon, 2,a\}$\\
$\{\alpha, 3,c\}$	&	$\{\beta, 3,d\}$ & $\{\gamma, 3,e\}$ & $\{\delta, 3,a\}$ & $\{\epsilon, 3,b\}$\\
$\{\alpha, 4,d\}$	&	$\{\beta, 4,e\}$ & $\{\gamma, 4,a\}$ & $\{\delta, 4,b\}$ & $\{\epsilon, 4,c\}$\\
$\{\alpha, 4,e\}$	&	$\{\beta, 5,a\}$ & $\{\gamma, 5,b\}$ & $\{\delta, 5,c\}$ & $\{\epsilon, 5,d\}$\\
\end{tabular}
\caption{$\m H_{15}$ has fifteen vertices that can be partitioned into three sets $\{1,2,3,4,5\}$, $\{a,b,c,d,e\}$ and $\{\alpha,\beta,\gamma,\delta,\epsilon\}$. It has four hundred and twenty-eight minimal transversals.}\label{tab:edges}
\end{figure}

\medskip
\begin{thm:lower}
There exists a constant $c$ such that for any integer $n\geq 15$, $$f_3(n)\geq c1.4977^n.$$
\end{thm:lower}
\begin{proof}

Let $n$ be an integer greater than 15. There are two unique integers $k$ and $r$ such that $r$ is in $[0,14]$ and $n = 15k+r$. We aim to build a hypergraph on $n$ vertices with many minimal transversals. The idea is to make the disjoint union of copies of $\m H_{15}$. To reach exactly $n$ vertices, the last copy is modified in the following way. Let $\m H_{15}^r$ be the hypergraph $\m H_{15}$ with $r$ more vertices $v_1,\dots,v_r$. In order for the edges to span all vertices, we add edges $\{v_i,1,a\}$ for all $i$ from $1$ to $r$. Observe that $\m H_{15}^r$ has at least as many minimal transversals as $\m H_{15}$. As a consequence of making the disjoint union of $k-1$ copies of $\m H_{15}$ and one copy of $\m H_{15}^r$, we obtain a tripartite 3-uniform hypergraph on $n$ vertices with at least $428^k$ minimal transversals. Since $k=\frac{n-r}{15}$ and $428^{\frac{1}{15}}>1.4977$, fixing $c=428^{-\frac{14}{15}}$ we conclude that for all $n$ greater than $15$
$$ f_3(n) \geq c\cdot 1.4977^n. $$

\end{proof}

Another way to see the hypergraph $\m H_{15}$ is with a 3-dimensional cross table, where a cross in cell $(\alpha, 1, a)$ represents the edge $\{\alpha, 1, a\}$.
%Note that the crosses representing the 3-edges of this particular hypergraph are a solution to the 3-dimensional chess rook problem.
%This remark is probably devoid of any profound meaning.

\begin{figure}[ht]
\resizebox{\textwidth}{!}{
\begin{tabular}{|c||ccccc||ccccc||ccccc||ccccc||ccccc|}
\hline
 & a & b & c & d & e & a & b & c & d & e & a & b & c & d & e & a & b & c & d & e & a & b & c & d & e \\
\hline
1 & $\times$ & & & & & & $\times$ & & & & & & $\times$ &  & & & & & $\times$ & & & & & & $\times$\\
2 & & $\times$ & & & & & & $\times$ & & & & & & $\times$ & & & & & & $\times$ & $\times$ & & & &\\
3 & & & $\times$ & & & & & & $\times$ & & & & & & $\times$ & $\times$ & & & & & & $\times$ & & &\\
4 & & & & $\times$ & & & & & & $\times$ & $\times$ & & & & & & $\times$ & & & & & & $\times$ & &\\
5 & & & & & $\times$ & $\times$ & & & & & & $\times$ & & & & & & $\times$ & & & & & & $\times$ &\\
\hline
\multicolumn{1}{|c||}{} & \multicolumn{5}{c||}{$\alpha$} & \multicolumn{5}{c||}{$\beta$} & \multicolumn{5}{c||}{$\gamma$} & \multicolumn{5}{c||}{$\delta$} & \multicolumn{5}{c|}{$\epsilon$}\\
\hline
\end{tabular}
}
\caption{Another representation of $\m H_{15}$. Each cross represents a 3-edge.}\label{tab:hyp}
\end{figure}

%%%%%%%%%%%%%%%%%%%%%%%%%%%%%%%%%%%%%%%%%%%%%%%%%%%%%%%%%%%%%%%%%
%%%%%%%%%%%%%%%%%%%%%%%%%%%%%%%%%%%%%%%%%%%%%%%%%%%%%%%%%%%%%%%%%
%%%%%%%%%%%%%%%%%%%%%%%%%%%%%%%%%%%%%%%%%%%%%%%%%%%%%%%%%%%%%%%%%
\section{Upper Bound}

In this section, we prove that $f_3(n)\leq 1.5012^n$ (Theorem~\ref{thm:upper}) by first proving the technical Lemma~\ref{lemma:number}. The general structure of the proof is similar to Lonc and Truszczy{\'n}ski's proof of their bound to the maximal number of minimal transversals in 3-uniform hypergraphs in~\cite{DBLP:journals/dm/LoncT08} and is as follows.

\medskip
 
%First, we define a rooted tree by forcing the presence or absence of vertices in a minimal transversal.
%In this tree, each internal node is a tripartite 3-uniform hypergraph, and each leaf corresponds to a minimal transversal of the root.
%Then, we count the number of leaves in this tree.
%We do that by using the so-called "$\tau$-lemma" introduced by Kullmann~\cite[Lemma 8.2]{DBLP:journals/tcs/Kullmann99}.
%We associate a measure to each hypergraph and use it to label the edges of our tree with a carefully chosen distance between the parent (hypergraph) and the child (hypergraph). The estimation of the number of leaves is then done by computing the maximal $\tau$, that depends on the measure for each hypergraph (for each inner node of the tree).
Given a rooted tree. If we have a probability distribution attributed to each internal node for picking one of its children, we may apply the following random process: start from the root, and follow the node distribution to pick one of its children until you reach a leaf. In this process, the probability to end up in some fixed leaf is the product of probabilities on the unique path from this leaf to the root. The sum over all leaves of these probabilities equals 1. Consider the smallest such probability $p$. Then the number of leaves is bounded above by $p^{-1}$. Actually, for any tree, there is such a probability distribution that assigns the same probability to each leaf in the end. When facing an unknown structure, we may not be able to find that optimum distribution, but this tool gives us some power to adjust the counting of the leaves.

\medskip
In this paper, we build a tree where the root is some specific hypergraph $\m H_0$. All nodes will be hypergraphs and the leaves can be identified with minimal transversals of $\m H_0$. We build an adequate probability distribution by designing some measure function $\mu$ on our hypergraphs. For internal node $\m H$, we shall fix some value $\tau$ and give probability $\tau^{\mu(\m H')-\mu(\m H)}$ to the transition towards its child $H'$. This value $\tau$ is uniquely defined since we want the sum to equal 1 at each step. In the end, if we manage to bound all values $\tau$ above by some $\tau_0$ and if the leaves have non-negative measures, the product of all those probabilities shall be bounded below by $\tau_0^{-\mu(\m H)}$ and thus the number of leaves is no more than $\tau_0^{\mu(\m H)}$.

\medskip

\begin{figure}[h]
\centering
\begin{tikzpicture}[every node/.style={fill,circle,inner sep=0pt,minimum size=4pt}]

\node[label =$a$] (a) at (0,0) {};
\node[label = $1$] (1) at (-1,1) {};
\node[label = $\gamma$] (alpha) at (-1,-1) {};

\node[label = $\delta$] (beta) at (-3,1) {};
\node[label = $2$] (2) at (-3,-1) {};
\node[label = b] (b) at (-2,0) {};

  \draw ($(a)+(0.3,0)$)
    to[out=90,in=0] ($(1)+(0,0.37)$)
    to[out=180,in=180] ($(alpha)+(0,-0.3)$)
    to[out=0,in=-90] ($(a)+(0.3,0)$);
    
  \draw ($(b)+(0,-0.2)$)
    to[out=0,in=-90] ($(1)+(0.3,0)$)
    to[out=90,in=90] ($(beta)+(-0.3,0)$)
    to[out=-90,in=180] ($(b)+(0,-0.2)$);
    
  \draw ($(b)+(0,0.37)$)
    to[out=0,in=90] ($(alpha)+(0.3,0)$)
    to[out=-90,in=-90] ($(2)+(-0.3,0)$)
    to[out=90,in=180] ($(b)+(0,0.37)$);
    
\end{tikzpicture}
\caption{This small hypergraph shall serve as an example to illustrate the following concepts of condition, procedure and measure.\label{fig:exempleH}}
\end{figure}

Let us dive into the more technical part, illustrated with Figure~\ref{fig:exempleH}'s example hypergraph.
We use the notion of \emph{condition}, introduced thereafter.

\medskip

\begin{definition}
Given a set $V$ of vertices, a \emph{condition} on $V$ is a pair $(A^+,A^-)$ of disjoint sets of vertices.
A condition is \emph{trivial} if $A^+\cup A^-=\emptyset$, and \emph{non-trivial} otherwise.
\end{definition}

\medskip

All the conditions that we handle are non-trivial.
A set $T$ of vertices \emph{satisfies} a condition $(A^+,A^-)$ if $A^+\subseteq T$ and $T\cap A^-=\emptyset$.
As such, having vertex sets satisfy a condition amounts to forcing a set of vertices to be present (the vertices in $A^+$) and forbidding other vertices (the vertices in $A^-$).
For instance, $\{1, b\}$ is a minimal transversal of the hypergraph depicted in Figure~\ref{fig:exempleH} that satisfies the condition $(\{1\}, \{\gamma\})$ because it contains $1$ but not $\gamma$.

%{\color{red} Mehhhhhhh cette phrase est l\`a, certes, mais elle sert \`a quoi ?}
%For example, let us consider the hypergraph from Figure~\ref{fig:exempleH}.
%We choose condition $A = (\{1\}, \{\alpha\})$.
%Then, the set of vertices $T = {1, b}$ satisfies condition $A$.

\medskip
Let $\m H$ be a hypergraph and $(A^+,A^-)$ a condition.
The hypergraph $\m H_{(A^+,A^-)}$ is constructed from $\m H$ and $(A^+,A^-)$ through the following procedure:

\begin{enumerate}
\item remove every edge that contains a vertex that is in $A^+$;
\item remove from every remaining edge the vertices that are in $A^-$;
\item remove redundant edges.
\end{enumerate}

Vertices of $\m H$ that appear in a condition $(A^+,A^-)$ are not in $\m H_{(A^+,A^-)}$ as they are either removed from all the edges or all the edges that contain them have disappeared. For instance, if $\m H$ is the hypergraph depicted in Figure~\ref{fig:exempleH} and $A=(\{1\}, \{\gamma\})$, the hypergraph $\m H_{(\{1\}, \{\gamma\})}$ is the one depicted in Figure~\ref{fig:exempleHcond}.

%If the hypergraph  $\m H$ is the one depicted in Figure~\ref{fig:exempleH} and the condition is $(\{1\}, \{\alpha\})$, in the resulting hypergraph $\m H_{(\{1\}, \{\alpha\})}$, every edge that contains ${1}$ will be removed as well as the vertex $\alpha$.
%It is represented in Figure~\ref{fig:exempleHcond}.

\begin{figure}[h]
\centering
\begin{tikzpicture}[every node/.style={fill,circle,inner sep=0pt,minimum size=4pt}]

\node[label = $2$] (2) at (-3,-1) {};
\node[label = b] (b) at (-2,0) {};

 \draw (b) -- (2);   

\end{tikzpicture}
\caption{Let $\m H$ be the hypergraph in Figure~\ref{fig:exempleH}. Then, this figure depicts $\m H_{(\{1\}, \{\gamma\})}$. The edges that contained $1$ were removed from the hypergraph and the vertex $\alpha$ was removed from the remaining edge.\label{fig:exempleHcond}
}
\end{figure}

\medskip

\begin{lemma}
\label{lemma:proc}
Let $\m H$ be a hypergraph, $(A^+,A^-)$ be a condition and $T$ be a set of vertices of $\m H$.
If $T$ is a minimal transversal of $\m H$ and $T$ satisfies $(A^+,A^-)$, then $T\setminus A^+$ is a minimal transversal of $\m H_{(A^+,A^-)}$.
\end{lemma}
\begin{proof}
The proof is straightforward from the construction of $\m H_{(A^+,A^-)}$.
\end{proof}

%This echoes the parent-child relation that can be found in the exact exponential-time algorithm field~\cite{fomin2010exact}.

\medskip
%{\color{red}Ici on parle de minimal transversal ou de transversal ?} Minimal, vu que V est une traverse aussi et satisfait boffement les conditions avec une négation
A family of conditions is \emph{complete} for the hypergraph $\m H$ if the family is non-empty, each condition in the family is non-trivial, and every minimal transversal of $\m H$ satisfies at least one condition of the family.
For instance, $\{(\{1\}, \{\gamma\})\}$ is not a complete family of conditions for the hypergraph depicted in Figure~\ref{fig:exempleH} because there is a minimal transversal, $\{\gamma, \delta\}$, that does not contain the vertex $1$ and thus does not satisfy any of the conditions. However, $\{(\{1\},\emptyset),(\emptyset,\{1\})\}$ is a complete family of conditions because all minimal transversals either contain or do not contain the vertex $1$.

\medskip

Let $\m C$ be the class of $k$-partite, with $k\leq 3$, hypergraphs such that their vertex set can be partitioned in such a way that one of the parts is a minimal transversal.
We suppose that they are implicitly partitioned in such a way and call $S$ the part that is a minimal transversal.
Tripartite 3-uniform hypergraphs belong to this class.
It is clear that $k$-partite hypergraphs do not become $(k+1)$-partite when vertices are removed from edges or edges are deleted.
As such, if a hypergraph $\m H=(V,\m E)$ belongs to the class $\m C$ and $A=(A^+,A^-)$ is a condition on $V$ such that the edges that contain vertices of $A^-\cap S$ also contain vertices of $A^{+}$, then $\m H_{A}$ is in $\m C$. From now on, we suppose that all the conditions we handle respect this property.

%{\color{red}TODO: Example 

%euhhhh je sais pas ce qu'on peut mettre en exemple ici}

\medskip

A hypergraph is non-trivial if it is not empty.
A \emph{descendant function} for $\m C$ is a function that assigns to each non-trivial hypergraph in $\m C$ a complete family of conditions.
Let $\rho$ be such a function.

%TODO: Est-ce qu'on peut vraiment dire qu'on "définit" l'arbre ? Est-ce qu'il est pas mécaniquement induit par \rho ?
Using $\rho$, we can construct a rooted labeled  tree $\mathcal T_\m H$ for all hypergraphs $\m H$ in $\m C$.
When $\m H$ is trivial, then $\mathcal T_\m H$ is a single node labeled with $\m H$.
When $\m H$ is non-trivial, we create a node labeled with $\m H$ and make it the parent of the root of all the trees $\mathcal T_{\m H_A}$, for $A\in\rho(\m H)$.
Note that this construction is possible because $\m C$ is closed under the operation of removing edges and removing vertices from edges, and the number of vertices can only decrease when the transformation from $\m H$ to $\m H_A$ occurs.
Such a rooted tree corresponding to our toy example is presented in Figure~\ref{fig:bigTree}.
One can check that all the conditions above are respected.
%TODO: reformuler cette dernière phrase...

\begin{figure}
\begin{tikzpicture}
    \node[draw] (root) at (0,-0.5) {
        \begin{tikzpicture}[every node/.style={fill,circle,inner sep=0pt,minimum size=4pt}]
        
        \node[label =$a$] (a) at (0,0) {};
        \node[label = $1$] (1) at (-1,0.8) {};
        \node[label = $\gamma$] (alpha) at (-1,-0.8) {};
        
        \node[label = $\delta$] (beta) at (-3,0.8) {};
        \node[label = $2$] (2) at (-3,-0.8) {};
        \node[label = b] (b) at (-2,0) {};
        
          \draw ($(a)+(0.3,0)$)
            to[out=90,in=0] ($(1)+(0,0.37)$)
            to[out=180,in=180] ($(alpha)+(0,-0.3)$)
            to[out=0,in=-90] ($(a)+(0.3,0)$);
            
          \draw ($(b)+(0,-0.2)$)
            to[out=0,in=-90] ($(1)+(0.3,0)$)
            to[out=90,in=90] ($(beta)+(-0.3,0)$)
            to[out=-90,in=180] ($(b)+(0,-0.2)$);
            
          \draw ($(b)+(0,0.37)$)
            to[out=0,in=90] ($(alpha)+(0.3,0)$)
            to[out=-90,in=-90] ($(2)+(-0.3,0)$)
            to[out=90,in=180] ($(b)+(0,0.37)$);
            
        \end{tikzpicture}};

    \coordinate (b1c1) at (-6,-7.5);
    \coordinate (b1c2) at (-4,-7.5);
    \coordinate (b1c3) at (-3,-7.5);
    
    \node[draw] (b1) at (-3,-5) {
        \begin{tikzpicture}[every node/.style={fill,circle,inner sep=0pt,minimum size=4pt}]
        
        \node[label = $\gamma$] (alpha) at (-1,-1) {};
        
        \node[label = $2$] (2) at (-3,-1) {};
        \node[label = b] (b) at (-2,0) {};
        
          \draw ($(b)+(0,0.37)$)
            to[out=0,in=90] ($(alpha)+(0.3,0)$)
            to[out=-90,in=-90] ($(2)+(-0.3,0)$)
            to[out=90,in=180] ($(b)+(0,0.37)$);
            
        \end{tikzpicture}};
    \draw (b1) -- node[midway, left] {$(\{b\},\emptyset)$} (b1c1);
    \draw (b1) -- node[midway, left] {$(\{2\},\emptyset)$} (b1c2);
    \draw (b1) -- node[midway, right] {$(\{\gamma\},\emptyset)$} (b1c3);
    
    \node[draw] (b2) at (3,-5) {
        \begin{tikzpicture}[every node/.style={fill,circle,inner sep=0pt,minimum size=4pt}]
        
        \node[label =$a$] (a) at (0,0) {};
        \node[label = $\gamma$] (alpha) at (-1,-0.7) {};
        
        \node[label = $\delta$] (beta) at (-3,0.8) {};
        \node[label = $2$] (2) at (-2.7,-0.7) {};
        \node[label = b] (b) at (-2,0) {};
        
          \draw (a)-- (alpha);
            
          \draw (b) -- (beta);
          
          \draw ($(b)+(0,0.37)$)
            to[out=0,in=90] ($(alpha)+(0.3,0)$)
            to[out=-90,in=-90] ($(2)+(-0.3,0)$)
            to[out=90,in=180] ($(b)+(0,0.37)$);
            
        \end{tikzpicture}};

    \node[draw] (b3) at (-1,-9) {
        \begin{tikzpicture}[every node/.style={fill,circle,inner sep=0pt,minimum size=4pt}]
        
        \node[label =$a$] (a) at (0,0) {};
        \node[label = $\gamma$] (alpha) at (-1,-1) {};
        
        \draw (a)-- (alpha);
        \end{tikzpicture}};
    
    \coordinate (b3c1) at (-2,-10.5);
    \coordinate (b3c2) at (-1,-10.5);
    \draw (b3) -- node[midway, left] {$(\{a\},\emptyset)$} (b3c1);
    \draw (b3) -- node[midway, right] {$(\{\gamma\},\emptyset)$} (b3c2);
    
    \node[draw] (b4) at (1.4,-9) {
        \begin{tikzpicture}[every node/.style={fill,circle,inner sep=0pt,minimum size=4pt}]
        
        \node[label = $\delta$] (beta) at (-3,1) {};
        
        \draw ($(beta)+(0,0.37)$)
            to[out=0,in=90] ($(beta)+(0.3,0)$)
            to[out=-90,in=0] ($(beta)+(0,-0.37)$)
            to[out=180,in=-90] ($(beta)+(-0.3,0)$)
            to[out=90,in=180] ($(beta)+(0,0.37)$);
            
        \end{tikzpicture}};
    \coordinate (b4c1) at (2,-10.5);
    \draw (b4) -- node[midway,right] {$(\{\delta\},\emptyset)$} (b4c1);
    
    \node[draw] (b5) at (5,-9) {
        \begin{tikzpicture}[every node/.style={fill,circle,inner sep=0pt,minimum size=4pt}]
        
        \node[label = $\delta$] (beta) at (-1,1) {};
        
        \node[label =$a$] (a) at (0,0) {};
        
        \node[label = $2$] (2) at (-1,-1) {};
        
        \draw ($(beta)+(0,0.37)$)
            to[out=0,in=90] ($(beta)+(0.3,0)$)
            to[out=-90,in=0] ($(beta)+(0,-0.37)$)
            to[out=180,in=-90] ($(beta)+(-0.3,0)$)
            to[out=90,in=180] ($(beta)+(0,0.37)$);
        
        \draw ($(a)+(0,0.37)$)
            to[out=0,in=90] ($(a)+(0.3,0)$)
            to[out=-90,in=0] ($(a)+(0,-0.37)$)
            to[out=180,in=-90] ($(a)+(-0.3,0)$)
            to[out=90,in=180] ($(a)+(0,0.37)$);
        
        \draw ($(2)+(0,0.37)$)
            to[out=0,in=90] ($(2)+(0.3,0)$)
            to[out=-90,in=0] ($(2)+(0,-0.37)$)
            to[out=180,in=-90] ($(2)+(-0.3,0)$)
            to[out=90,in=180] ($(2)+(0,0.37)$);    
        \end{tikzpicture}};
    \coordinate (b5c1) at (5,-11.5);
    \draw (b5) -- node[midway, left] {$(\{a2\delta\},\emptyset)$} (b5c1);

    \draw (root) -- node[midway, left] {$(\{1\},\emptyset)$} (b1);
    \draw (root) -- node[midway, right] {$(\emptyset,\{1\})$} (b2);
    \draw (b2) -- node[midway, left] {$(\{b\},\emptyset)$}  (b3);
    \draw (b2) -- node[midway,right] {$(\{\gamma\},\{b\})$} (b4);
    \draw (b2) -- node[midway, right] {$(\emptyset,\{\gamma b\})$} (b5);
\end{tikzpicture}
\caption{Example of a decomposition tree for our running example. The measure $\mu$ is given for each hypergraph, while the conditions are given on the edges.\label{fig:bigTree}}
\end{figure}

%{\color{red}c'est de nous cette prop ? c'était formulé comme ça de base ?

%C'est pas de nous, et c'est pas formulé comme ça j'imagine}
\begin{proposition}
Let $\rho$ be a descendant function for a hypergraph class closed under removing edges and removing vertices from edges.
Then for all hypergraphs $\m H$ in such a class, the number of minimal transversals is bounded above by the number of leaves of $\mathcal T_\m H$.
\end{proposition}
\begin{proof}
If $\m H$ is trivial, then it has only one transversal, the empty set.
If the empty set is an edge of $\m H$, then $\m H$ has no transversals.
In both cases, the proposition follows directly from the definition of the tree.
Let us assume now that $\m H$ is non-trivial, and that the proposition is true for all hypergraphs with fewer vertices than $\m H$.

As $\m H$ is non-trivial, $\rho$ is well defined for $\m H$ and $\rho(\m H)$ is a complete family of conditions.
Let $X$ be a minimal transversal of $\m H$.
Then $X$ satisfies at least one condition $A$ in $\rho(\m H)$.
From Lemma~\ref{lemma:proc}, we know that there is a minimal transversal $Y$ of $\m H_A$ such that $Y= X\setminus A^+$.
Then the number of minimal transversals of $\m H$ is at most the sum of the number of minimal transversals in its children.
\end{proof}

\medskip

Bounding the number of leaves in $\mathcal T_\m H$ for all hypergraphs $\m H\in \m C$ is thus bounding $f_3(n)$.
In order to do that, we use Lemma~\ref{lemma:b1} proven by Kullmann~\cite{DBLP:journals/tcs/Kullmann99}.
We denote by $L(\mathcal T)$ the set of leaves of a rooted tree $\mathcal T$ and, for a leaf $\ell\in L(\mathcal T)$, we denote by $P(\ell)$ the set of edges on the path from the root to $\ell$.

\medskip

%{\color{red}TODO: est-ce que c'est recopié de Kullmann verbatim ? sinon le "(that is... )" est bof.

%Pas verbatim, on peut changer ce qu'on veut. c'est juste que transition probability soit c'était défini ailleurs chez kullmann soit jamais défini, donc on avait du ajouter ça là pour pas faire un autre environnement}
\begin{lemma}[{\cite[Lemma 8.1]{DBLP:journals/tcs/Kullmann99}}]
\label{lemma:b1}
Consider a rooted tree $\m T$ with an edge labeling $w$ with value in the interval $[0,1]$ such that for every internal node, the sum of the labels on the edges from that node to its children is 1 (that is a transition probability).

Then, $$|L(\mathcal T)|\leq \max_{\ell\in L(\mathcal T)}\left(\prod_{e\in P(\ell)}w(e)\right)^{-1}.$$
\end{lemma}

\medskip

In order to pick an adequate probability distribution, we use a measure.
A \emph{measure} $\mu$ is a function that assigns to any hypergraph $\m H$ in $\m C$ a real number $\mu(\m H)$ such that $0\leq \mu(\m H)\leq |V(\m H)|$. Let $A$ be a condition on the vertices of hypergraph $\m H$ and $\mu$ be a measure. We define
\[\Delta(\m H,\m H_A)=\mu(\m H)-\mu(\m H_A).\]

\medskip

%Let $\m H$ be a hypergraph in $\m C$ and $\mu$ a measure. 
If, for every condition $A$ in $\rho(\m H)$, $\mu(\m H_A)\leq\mu(\m H)$, then we say that $\rho$ is $\mu$-compatible.
In this case, there is a unique positive real number $\tau\geq1$ such that 

\begin{equation*}
\sum_{A\in\rho(\m H)}\tau^{-\Delta(\m H,\m H_A)}=1.
\end{equation*}
%TODO: ça a l'air de sortir de nous. Faudrait ptete dire que ça vient de Kullmann ? Personne va aller faire chier Kullmann.

\medskip

When $\tau\geq 1$, $\sum_{A\in\rho(\m H)}\tau^{-\Delta(\m H,\m H_A)}$ is a strictly decreasing continuous function of $\tau$.
For $\tau=1$, it is at least 1, since $\rho(\m H)$ is not empty, and it tends to 0 when $\tau$ tends to infinity.

A descendant function defined on a class $\m C$ is $\mu$-bounded by $\tau_0$ if, for every non-trivial hypergraph $\m H$ in $\m C$, $\tau\leq \tau_0$.

\medskip

Now, we adapt the $\tau$-lemma proven by Kullmann~\cite{DBLP:journals/tcs/Kullmann99} to our formalism.

\begin{theorem}[Kullmann~\cite{DBLP:journals/tcs/Kullmann99}]
\label{thm:expo}
Let $\mu$ be a measure and $\rho$ a descendant function, both defined on a class $\m C$ of hypergraphs closed under the operations of removing edges and removing vertices from edges.
If $\rho$ is $\mu$-compatible and $\mu$-bounded by $\tau_0$, then for every hypergraph $\m H$ in $\m C$, $$|L(\mathcal T_\m H)|\leq \tau_0^{h(\mathcal T_\m H)}$$ where $h(\mathcal T_\m H)$ is the height of $\mathcal T_\m H$.
\end{theorem}

\medskip

%This theorem comes from Lemma~\ref{lemma:b1}.
%We find the worst possible branching and assume that we apply it on every node on the path from the root to the leaves.
We now use Theorem~\ref{thm:expo} to provide an upper bound to $f_3(n)$. We start by proving Lemma~\ref{lemma:number} using an approach similar to that of~\cite{DBLP:journals/dm/LoncT08}.

\medskip

\begin{lemma}
\label{lemma:number}
There is a measure $\mu$ defined for every hypergraph $\m H$ in $\m C$ and a descendant function $\rho$ for $\m C$ that is $\mu$-compatible and $\mu$-bounded by $1.8393$.
\end{lemma}

\begin{proof}
%The general idea of this proof is that we identify a measure $\mu$ and a $\mu$-compatible descendant function $\rho$, i.e. a complete family of conditions for every hypergraph of $\m C$, that is $\mu$-bounded by $1.8393$.

%\medskip
Let $\m H$ be a hypergraph belonging to the class $\m C$, i.e., a $k$-partite, $k\leq 3$ hypergraph that contains a set S of vertices that is a minimal transversal such that no two vertices of $S$ belong to a same edge.

\medskip
We choose, as the measure $\mu(\m H)$, 

\[\mu(\m H) = |V(\m H)| - \alpha m(\m H)\]
\noindent
where $m(\m H)$ is the maximum number of pairwise disjoint 2-element edges in $\m H$ (i.e. the size of its maximum matchings) and $\alpha=0.145785$. The same measure is used in~\cite{DBLP:journals/dm/LoncT08} (with a different $\alpha$).

\medskip

We use Theorem~\ref{thm:expo} to bound the number of leaves in the tree $\mathcal T_\m H$ and thus the number of minimal transversals in $\m H$.
To do so, we define a descendant function $\rho$ that assigns a family of conditions to $\m H$ depending on its structure.
This takes the form of a case analysis.

\medskip
In each case, we consider a vertex $a\in S$ and its neighbours. The conditions are chosen in such a way that they are pairs $(A^+,A^-)$ of subsets of these neighbours and are sometimes strengthened to contain $a$ if and only if its presence in $A^+$ or $A^-$ is implied by our hypotheses. This causes every condition $A$ to respect the property that $A^+$ intersects all the edges containing a vertex of $A^-\cap S$, which lets $\m H_{A}$ remain in the class $\m C$.

\medskip
In each case $i$ and for every condition $A\in \rho(\m H)$, we find a bound $k_{\m H, A}$ such that

\[k_{\m H, A}\leq\Delta(\m H, \m H_A)\]

\noindent
and a unique positive real number $\tau_i$ that satisfies the equation

\begin{equation}
\label{eq:sumToOne}
\sum_{A\in \rho(\m H)} \tau_i^{-k_{\m H,A}}=1
\end{equation}

We show that $\tau_i\leq 1.8393$ for all $i$.
Let $\tau_0 = 1.8393$.

\medskip
As all our conditions involve at least one element from $V(\m H)\setminus S$, the height of $\mathcal T_\m H$ is bounded by $|V(\m H)|-|S|$. Hence, we have $$|L(\mathcal T_\m H)|\leq \tau_0^{|V(\m H)|-|S|}$$

\medskip
In the remainder of the proof, we will write conditions as sets of expressions of the form $a$ and $\overline b$ where $a$ means that $a$ is in $A^+$ and $\overline b$ means that $b$ is in $A^-$.
For example, the condition $(\{a,c\},\{b,d,e\})$ will be written as $ac\overline{bde}$ and the condition $(\{b,c\},\emptyset)$ will be written as $bc$.
For a vertex $v$, we denote by $d_2(v)$ the number of 2-edges that contain $v$, and by $d_3(v)$ the number of 3-edges that contain $v$.

\medskip
For each case, we suppose that the previous ones do not apply.

\medskip

\textbf{Case 1:} $d_2(a)\geq 2$ : the hypergraph $\m H$ contains a vertex $a$ from $S$ that belongs to at least two 2-edges $\{a,b\}$ and $\{a,c\} (Figure~\ref{fig:case1})$.

\begin{figure}[H]
    \centering
\begin{tikzpicture}
    \node[inner sep = 1pt] (a) at (0,0) {$a$};
    \node[inner sep = 1pt] (b) at (1,1) {$b$};
    \node[inner sep = 1pt] (c) at (1,-1) {$c$};

    \coordinate (coord0) at (0,-1);
    \coordinate (coord1) at (-0.8,1.2);
    \coordinate (coord2) at (-0.8,-0.5);

    \draw (a) -- (b);
    \draw (a) -- (c);
    \draw[dashed] (a) -- (coord0);

    \draw[dashed] ($(coord2)+(-0.3,-0.2)$)
        to[out=0,in=-90] ($(a)+(0.3,0)$)
        to[out=90,in=0] ($(coord1)+(-0.3,0.2)$);
    \end{tikzpicture}
    \caption{Case 1: The vertex $a\in S$ is in at least two 2-edges and may be part of some other 2-edges or 3-edges.\label{fig:case1}}
\end{figure}
A minimal transversal of $\m H$ either contains or does not contain $b$, and as such $\{b,\overline b\}$ is a complete family of conditions for $\m H$.
Similarly, a minimal transversal of $\m H$ either contains or does not contain $c$ so $\{bc,b\overline c,\overline b\}$ is also a complete family of conditions.
Minimal transversals of $\m H$ that do not contain $b$ or $c$ necessarily contain $a$ (as $\{a,b\}$ or $\{a,c\}$ would not be covered otherwise).
Hence $\{bc,ab\overline c, a\overline b\}$ is a complete family of conditions for $\m H$.

\medskip

Let $M$ be a maximum set of pairwise distinct 2-edges (matching) of $\m H$.
By removing $k$ vertices we decrease the size of a maximum matching by at most $k$. Hence, $|V(\m H_{A})| \leq |V(\m H)| - 2$ and $m(\m H_A)\geq m(\m H)-2$ when $A\in \{bc, a\overline b\}$. Similarly, $|V(\m H_{ab\overline c})| \leq |V(\m H)| - 3$ and $m(\m H_{ab\overline c})\geq m(\m H)-3$.
Thus, we have

\begin{equation}
\Delta(\m H, \m H_A)\geq\begin{cases}
2-2\alpha & \text{for~} A\in\{bc, a\overline b\}\\
3-3\alpha & \text{for~} A = ab\overline c
\end{cases}.
\end{equation}

Equation~\eqref{eq:sumToOne} becomes $2\tau_1^{2\alpha -2}+\tau_1^{3\alpha -3}=1$.
For our chosen $\alpha$, we have that $\tau_1\leq\tau_0$.

\medskip

\textbf{Case 2:} $d_2(a)=1$ :
the hypergraph $\m H$ contains a vertex $a$ from $S$ that belongs to a unique 2-edge $\{a,b\}$.
We break down this case into two sub-cases depending on whether or not $a$ belongs to some 3-edges: $d_3(a)=0$ and $d_3(a)\geq 1$.

\begin{figure}[H]
    \centering
\begin{tikzpicture}
    \node[inner sep = 1pt] (a) at (0,0) {$a$};
    \node[inner sep = 1pt] (b) at (1,1) {$b$};

    \coordinate (coord1) at (-0.8,1.2);
    \coordinate (coord2) at (-0.8,-0.5);

    \draw (a) -- (b);

    \draw[dashed] ($(coord2)+(-0.3,-0.2)$)
        to[out=0,in=-90] ($(a)+(0.3,0)$)
        to[out=90,in=0] ($(coord1)+(-0.3,0.2)$);
    \end{tikzpicture}
    \caption{Case 2: The vertex $a$ is part of only one 2-edge $\{a,b\}$ and may be part of other 3-edges.\label{fig:case2}}
\end{figure}

Since $a$ is in only one 2-edge, removing both $a$ and $b$ decreases the size of a maximum matching by at most 1.

%{\color{red} TODO: ptete mettre un schéma par sous-cas
%j'sais pas, ça semble pas utile ?}
\begin{itemize}
\item $d_3(a)=0$ : $a$ is in a single 2-edge $\{a,b\}$ and no 3-edges.
A minimal transversal of $\m H$ either contains or does not contain $b$.
As such, $\{b,\overline b\}$ is a complete family of conditions for $\m H$.
As $\{a,b\}$ is the only edge containing $a$, a minimal transversal of $\m H$ that contains $b$ cannot contain $a$.
Similarly, every minimal transversal of $\m H$ that does not contain $b$ necessarily contains $a$.
This makes $\{b\overline a, a\overline b\}$ a complete family of conditions for $\m H$.

Let $M$ be a maximum set of pairwise disjoint 2-edges of $\m H$. As $\{a,b\}$ is the only edge containing $a$, $b$ belongs to one of the edges in $M$. The hypergraphs $\m H_{b\overline a}$ and $\m H_{a \overline b}$ contain all the edges in $M$ except for the one containing $b$. Thus, $m(\m H_{b\overline a}) = m(\m H_{a \overline b}) \geq m(\m H) - 1$. Since $|V(\m H_{b\overline a})| = |V(\m H_{a \overline b})| \leq |V(\m H)| - 2$, we have

\begin{equation}
\Delta(\m H, \m H_A)\geq2-\alpha \text{~for~} A\in\{b\overline a,a\overline b\}.
\end{equation}
Equation~\eqref{eq:sumToOne} becomes $2\tau_{2.1}^{\alpha-2}=1$.
For our chosen $\alpha$, we have that $\tau_{2.1}\leq\tau_0$.

\begin{figure}[H]
    \centering
\begin{tikzpicture}
    \node[inner sep = 1pt] (a) at (0,0) {$a$};
    \node[inner sep = 1pt] (b) at (1,1) {$b$};
    \node[inner sep = 1pt] (C) at (1,-1) {$c$};
    \node[inner sep = 1pt] (d) at (0,-1) {$d$};

    \coordinate (coord1) at (-0.8,1.2);
    \coordinate (coord2) at (-0.8,-0.5);

    \draw (a) -- (b);

    \draw[dashed] ($(coord2)+(-0.3,-0.2)$)
        to[out=0,in=-90] ($(a)+(0.3,0)$)
        to[out=90,in=0] ($(coord1)+(-0.3,0.2)$);
    \draw ($(a)+(0,0.3)$)
        to[out=0, in=90] ($(c)+(0.3,0)$)
        to[out = -90, in = -90] ($(d)+(-0.3,0)$)
        to[out = 90, in = 180] ($(a)+(0,0.3)$);
        
    \end{tikzpicture}
    \caption{Case 2.2: The vertex $a$ is part of only one 2-edge $\{a,b\}$ and at least one 3-edge $\{a,c,d\}$.\label{fig:case22}}
\end{figure}

\item $d_3(a)\geq 1$ : $a$ is in a single 2-edge and in some 3-edges, one of which being $\{a,c,d\}$ (Figure~\ref{fig:case22}).
We start with the complete family of conditions $\{bc,bd\overline c, b\overline{cd}, \overline b\}$.
Any minimal transversal of $\m H$ that does not contain either $b$ or both $c$ and $d$ necessarily contains  $a$.
This makes $\{bc,bd\overline c, ab\overline cd, a\overline b\}$ a complete family of conditions for $\m H$.
Removing either $b$, $c$ or $d$ from the hypergraph decreases the size of the maximum matching by at most one. As mentioned previously, removing both $a$ and $b$ cannot decrease the size of the maximum matching by more than 1. As such, $m(\m H_{A}) \geq m(\m H) - |A|$ when $A\in \{bc, bd\overline c, a\overline b\}$ and $m(\m H_{ab\overline{cd}}) \geq m(\m H) - 3$.
We obtain 
\begin{equation}
\Delta(\m H,\m H_A)\geq\begin{cases}
2-2\alpha & \text{if~} A=bc\\
3-3\alpha & \text{if~} A=bd\overline c\\
4-3\alpha & \text{if~} A=ab\overline{cd}\\
2-\alpha & \text{if~} A=a\overline b
\end{cases}.
\end{equation}
Equation~\eqref{eq:sumToOne} becomes $\tau_{2.2}^{2\alpha-2}+\tau_{2.2}^{3\alpha-3}+\tau_{2.2}^{3\alpha-4}+\tau_{2.2}^{\alpha-2}=1$.
For our chosen $\alpha$, we have that $\tau_{2.2}\leq\tau_0$.

\end{itemize}

\medskip

\textbf{Case 3:} $d_2(a)=0$ and $d_3(a)\geq 1$ : the hypergraph $\m H$ contains a vertex $a$ from $S$ that is in  no 2-edge and in some 3-edges, one of which being $\{a,b,c\}$ (Figure~\ref{fig:case3}).

\begin{figure}[H]
\centering
\begin{tikzpicture}
\node (a) at (0,0) {$a$};

\node (b) at (-0.8,1.2) {$b$};
\node (c) at (-0.8,-0.5) {$c$};

\coordinate (coord3) at (0.8,1.2);
\coordinate (coord4) at (0.8,-0.5);

  \draw ($(c)+(0,-0.2)$)
    to[out=0,in=-90] ($(a)+(0.3,0)$)
    to[out=90,in=0] ($(b)+(0,0.2)$)
    to[out=180, in = 180] ($(c)+(0,-0.2)$);
    
  \draw[dashed] (coord4)
    to[out=180,in=-90] ($(a)+(-0.3,0)$)
    to[out=90,in=180] (coord3);
\end{tikzpicture}
\caption{Case 3: The vertex $a$ is in at least one 3-edge, but in no 2-edge.\label{fig:case3}}
\end{figure}

We start with the conditions $\{b,c\overline b, \overline{bc}\}$. Any minimal transversal that does not contain $b$ and $c$ necessarily contains $a$ so we strengthen the family of conditions to $\{b,c\overline b,a\overline{bc}\}$.
Since we do not have any 2-edge anymore (because previous cases do not apply), we cannot decrease the size of a maximum matching.
We obtain 
\begin{equation}
\Delta(\m H, \m H_A)\geq\begin{cases}
1 & \text{if~} A=b\\
2 & \text{if~} A=c\overline b\\
3 & \text{if~} A=a\overline{bc}
\end{cases}.
\end{equation}

Equation~\eqref{eq:sumToOne} becomes $\tau_{3}^{-1}+\tau_{3}^{-2}+\tau_{3}^{-3}=1$.
For our chosen $\alpha$, we have that $\tau_{3}\leq\tau_0$.
\end{proof}

\medskip

This proof ensures that there is a measure $\mu$ and a descendant function $\rho$ for our class of hypergraphs such that $\rho$ is $\mu$-bounded by $1.8393$.
This allows us to formulate the following theorem.

\begin{theorem}
\label{thm:number2}
The number of minimal transversals in a hypergraph belonging to the class $\m C$ is less than $1.8393^{n-|S|}$.
\end{theorem}

\begin{proof}
Let $\mu$ and $\rho$ be the measure and descendant function used in Lemma~\ref{lemma:number}'s proof. The height $h(\mathcal T_\m H)$ of the tree is less than $n-|S|$ so applying Theorem~\ref{thm:expo} yields the result.
\end{proof}

\medskip
Theorem~\ref{thm:upper} is a straightforward corollary of Theorem~\ref{thm:number2}.

\begin{thm:up}For any integer $n$, $$f_3(n)\leq 1.8393^{2n/3}\leq 1.5012^n.$$
\end{thm:up}
\begin{proof}
The vertices of a tripartite 3-uniform hypergraph can be partitioned into three minimal transversals so any of them can be $S$. The minimization of the bound is achieved by using the biggest set which, in the worst case, has size $n/3$.
\end{proof}

%%%%%%%%%%%%%%%%%%%%%%%%%%%%%%%%%%%%%%%%%%%%%%%%%%%%%%%%%%%%%%%%%
%%%%%%%%%%%%%%%%%%%%%%%%%%%%%%%%%%%%%%%%%%%%%%%%%%%%%%%%%%%%%%%%%
%%%%%%%%%%%%%%%%%%%%%%%%%%%%%%%%%%%%%%%%%%%%%%%%%%%%%%%%%%%%%%%%%
\section{Discussion and Conclusion}

%TODO : Bien réécrire la conclusion

%On a montré que ***
%these two bounds can be used in the study of the complexity of algorithms for mining 3-dimensional Boolean data (for example)
%These bounds can be improved by....
%Data mining is not limited to 3-dimensional data and nowadays there is growing interest in n-dimensional data so it would be interesting to...

In this paper, we showed that the maximum number of minimal transversals in $(3,3)$-hypergraphs of order $n$ is between $c1.4977^n$, where $c$ is a constant, and $1.5012^n$. Both bounds can be used to better analyse the worst-case complexity of algorithms for mining 3-dimensional Boolean data, such as TRIAS~\cite{jaschke2006trias}, Data-Peeler~\cite{cerf2008data} or the one proposed by Makhalova and Nourine~\cite{makhalova2017incremental}.

\medskip
As a future work, the upper bound could be improved through the same approach by choosing a better measure or branching. The lower bound could maybe be improved by finding other worst cases through computer search. The worst cases we managed to find all corresponded to solutions of the chess rook problem in $\frac{n}{3}\times\frac{n}{3}\times\frac{n}{3}$ matrices up to $n=5$ (see Figure~\ref{tab:hyp}). However, this did not seem to be the case for $n=6$ so we postulate that it does not work for $n>5$.

\medskip
There is also a growing interest for pattern mining in $k$-dimensional datasets with $k>3$ as reality cannot always be represented by mere ternary relations.
For this reason, it would be interesting to devise a more general proof, following the same schema, to bound the maximal number of minimal transversals in $(k,k)$-hypergraphs.

\section*{Acknowledgments}
Alexandre Bazin was supported by the European Union's EUROSTAR PSDP project. Giacomo Kahn was supported by the European Union's ``\emph{Fonds Europ\'een de D\'eveloppement R\'egional (FEDER)}'' program though project AAP ressourcement  S3 -- DIS 4 (2015-2018). Kaveh Khoshkhah was supported by the Estonian Research Council, ETAG (Eesti Teadusagentuur), through PUT Exploratory Grant \#620.

\section*{Bibliography}
\bibliographystyle{unsrt}
\bibliography{Biblio}

\end{document}